\documentclass[12pt,a4paper]{amsart}
\usepackage{amsfonts,amsmath,amssymb}
\usepackage{hyperref}
\usepackage{latexsym,fullpage,amsfonts,amssymb,amsmath,amscd,graphics,epic}
\usepackage[all]{xy}
\usepackage{amssymb,amsthm,amsxtra}
\usepackage[usenames]{color}
\usepackage{amscd}
\usepackage{amsthm}
\usepackage{amsfonts}
\usepackage{amssymb}
\usepackage{mathrsfs}
\usepackage{url}
\usepackage{bbm}
\usepackage{wasysym}
\usepackage{enumitem}
\usepackage[vcentermath]{youngtab}%
\theoremstyle{plain}
\newtheorem*{theorem*}{Theorem}
\newtheorem*{remark*}{Remark}
\newtheorem*{example*}{Example}
\newtheorem{lemma}{Lemma}[subsection]
\newtheorem{proposition}[lemma]{Proposition}
\newtheorem{corollary}[lemma]{Corollary}
\newtheorem{theorem}[lemma]{Theorem}

\newtheorem*{conjecture*}{Conjecture}

\theoremstyle{definition}
\newtheorem{definition}[lemma]{Definition}

\newtheorem{example}[lemma]{Example}

\theoremstyle{remark}
\newtheorem{remark}[lemma]{Remark}

\newtheorem{notation}[lemma]{Notation}

 \newcommand{\idealI}{\mathfrak{I}}
\newcommand{\Hom}{\operatorname{Hom}}

\newcommand{\End}{\operatorname{End}}

\newcommand{\bC}{{\mathbb C}}
\newcommand{\bZ}{{\mathbb Z}}

\newcommand{\lam}{{\lambda}}

\newcommand{\fh}{{\mathfrak{h}}}

\newcommand{\T}{\tau}

\newcommand{\abs}[1]{\left|{#1}\right|}

\newcommand{\Dab}{\underline{Rep}^{ab}(S_{t})}

\newcommand{\InnaA}[1]{{{#1}}}
\newcommand{\InnaB}[1]{{{#1}}}
\def\quotient#1#2{%
    \raise1ex\hbox{$#1$}\Big/\lower1ex\hbox{$#2$}%
}

\begin{document}

\date{\today}
\title{Deligne categories and reduced Kronecker coefficients}
 \author{Inna Entova Aizenbud}
\address{Inna Entova Aizenbud,
Massachusetts Institute of Technology,
Department of Mathematics,
Cambridge, MA 02139 USA.}
\email{inna.entova@gmail.com}

\begin{abstract}
 The Kronecker coefficients are the structural constants for the tensor categories of representations of the symmetric groups, namely, given three partitions $\lam, \mu, \T$ of $n$, the multiplicity of $\lambda$ in $\mu \otimes \T$ is called the Kronecker coefficient $g^{\lam}_{\mu, \T}$.
 
 When the first part of each of the partitions is taken to be very large (the remaining parts being fixed), the values of the appropriate Kronecker coefficients stabilize; the stable value is called the reduced (or stable) Kronecker coefficient. These coefficients also generalize the Littlewood-Richardson coefficients, and have been studied quite extensively. 
 
 In this paper, we show that reduced Kronecker coefficients appear naturally as structure constants of Deligne categories $\underline{Rep}(S_t)$. This allows us to interpret various properties of the reduced Kronecker coefficients as categorical properties of Deligne categories $\underline{Rep}(S_t)$.
\end{abstract}
\maketitle
\setcounter{tocdepth}{3}
\section{Introduction}
The Kronecker coefficients are the structural constants for the semisimple categories $Rep(S_n)$. Namely, considering two irreducible representations $\mu, \T$ of $S_n$, we can decompose the tensor product $\mu \otimes \T$ into a direct sum of irrreducible representations of $S_n$. The multiplicity of $\lambda$ in $\mu \otimes \T$ is called {\it the Kronecker coefficient $g^{\lam}_{\mu, \T}$} (a good reference for Kronecker coefficients is \cite[Chapter 1, Par. {\rm V}]{Mac}).

Consider three arbitrary Young diagrams $\lam, \mu, \T$. For $n>>0$, denote by $\widetilde{\lam}(n)$ the Young diagram of size $n$ obtained by adding a top row of size $n-\abs{\lam}$ to $\lam$ (similarly for $\mu, \T$). It was noticed by Murnaghan in \cite{Mu1} that the sequence $\{g^{\widetilde{\lam}(n)}_{\widetilde{\mu}(n), \widetilde{\T}(n)} \}_{n>>0}$ stabilizes, and the stable value of the sequence was called the {\it reduced Kronecker coefficient $\bar{g}^{\lam}_{\mu, \T}$} associated with the triple $(\lam, \mu, \T)$. 

The reduced Kronecker coefficients were subsequently studied in \cite{Mu2}, \cite{BOR}, \cite{BDO}, and other papers.

It turns out that reduced Kronecker coefficients occur naturally in Deligne categories $\underline{Rep}(S_{t}), t \in \bC$, which are interpolations of the categories of finite-dimensional representations of the symmetric groups over the field $\bC$ of complex numbers. These Karoubian rigid symmetric monoidal categories were defined and studied by P. Deligne in \cite{D}, and subsequently by J. Comes and V. Ostrik in \cite{CO}, \cite{CO2}.

A detailed description of the Deligne categories $\underline{Rep}(S_{t})$, as well as their abelian envelopes, is given is Section \ref{sec:Del_cat_S_nu}. Below we give the main properties of these categories.

As it was said before, the categories $\underline{Rep}(S_{t})$ interpolate the categories of representations of symmetric groups; namely, for $n \in \bZ_+$, the category $\underline{Rep}(S_{t=n})$ admits a full, essentially surjective symmetric monoidal functor to the category of finite-dimensional representations of the symmetric group $S_n$.

For $t \not\in \bZ_+$, the category $\underline{Rep}(S_{t})$ is a semisimple abelian tensor category, the simple objects being parameterized by arbitrary Young diagrams (of any size). 

Note that for $n \in \bZ_+$, the category $\underline{Rep}(S_{t=n})$ is not abelian. It can be embedded as a full monoidal subcategory into a tensor (i.e. rigid symmetric monoidal abelian) category $\underline{Rep}^{ab}(S_{t=n})$; the latter is not semisimple, but its structure can be described quite explicitly (c.f. \cite[Section 5]{EA}).

For a generic value of $t$, the structural constants of $\underline{Rep}(S_{t})$ as a tensor category turn out to be exactly the reduced Kronecker coefficients; that is, denoting by $X_{\lambda}$ the simple object of $\underline{Rep}(S_{t})$ ($t \not\in \bZ_+$) which corresponds to the Young diagram $\lambda$, we get:

$$ X_{\mu} \otimes X_{\T} = \bigoplus_{\lambda \text{ is a partition of arbitrary size}} \bC^{\bar{g}^{\lambda}_{\mu, \tau}} \otimes X_{\lambda}$$

This approach provides a natural environment for the reduced Kronecker coefficients. It allows us to interpret various facts about reduced Kronecker coefficients in terms of Deligne categories, and obtain new, previously unknown formulas.

\subsection{Structure of the paper}
In Section \ref{sec:Del_cat_S_nu}, we remind the relevant facts about Deligne's category $\underline{Rep}(S_t)$.

In Section \ref{sec:red_Kron_coeff_and_Deligne_cat}, we define the reduced Kronecker coefficients in temrs of Deligne categories.

In Section \ref{sec:prop_Kron_coeff} we prove some known properties of Kronecker coefficients using the machinery of Deligne categories; we also prove some previously unknown formulas in Subsection \ref{ssec:new_form_Kron}.

\section{Notation and definitions}\label{sec:notation}
The base field throughout the paper will be $\bC$.
\subsection{Karoubian categories}
\begin{definition}[Karoubian category]
We will call a category $\mathcal{A}$ Karoubian\footnote{Deligne calls such categories "pseudo-abelian" (c.f. \cite[1.9]{D}). } if it is an additive category, and every idempotent morphism is a projection onto a direct factor.
\end{definition}

\begin{definition}[Block of a Karoubian category]
 A block in an Karoubian category is a full subcategory generated by an equivalence class of indecomposable objects, defined by the minimal equivalence relation such that any two indecomposable objects with a non-zero morphism between them are equivalent.
\end{definition}

\subsection{Symmetric group and Young diagrams}
\begin{notation}
\mbox{}
 \begin{itemize}[leftmargin=*]
\item $S_n$ will denote the symmetric group ($n \in \bZ_+$).

 \item The notation $\lambda$ will stand for a partition (weakly decreasing sequence of non-negative integers), a Young diagram $\lambda$, and the corresponding irreducible representation of $S_{\abs{\lambda}}$. Here $\abs{\lambda}$ is the sum of entries of the partition, or, equivalently, the number of cells in the Young diagram $\lambda$.

 \item \InnaA{All the Young diagrams will be considered in the English notation, i.e. the lengths of the rows decrease from top to bottom.}

\item The length of the partition $\lambda$, i.e. the number of rows of Young diagram $\lambda$, will be denoted by $\ell(\lambda)$.

\item The $i$-th entry of a partition $\lambda$, as well as the length of the $i$-th row of the corresponding Young diagram, will be denoted by $\lambda_i$ (if $i>\ell(\lambda)$, then $\lambda_i :=0$). 
%

\item $\fh$ (in context of representations of $S_n$) will denote the permutation representation of $S_n$, i.e. the $n$-dimensional representation $\bC^n$ with $S_n$ acting by $g.e_j =e_{g(j)}$ on the standard basis $e_1, .., e_n$ of $\bC^n$.

%

\item \InnaA{For any Young diagram $\lambda$ and an integer $n$ such that $n \geq \abs{\lambda} +\lambda_1$, we denote by $\widetilde{\lambda}(n)$ the Young diagram obtained by adding a row of length $n - \abs{\lambda}$ on top of $\lambda$.}

\end{itemize}

\end{notation}
\InnaA{

\begin{example}
Consider the Young diagram $\lambda$ corresponding to the partition $(6,5,4,1)$:
 $$\yng(6,5,4,1) $$ 
 The length of $\lambda$ is $4$, \InnaB{and $\abs{\lambda}=16$}. For $n = 23$, we have:
 $$\widetilde{\lambda}(n) = \yng(7,6,5,4,1)$$
\end{example}
}

\section{Deligne category $\underline{Rep}(S_{t})$}\label{sec:Del_cat_S_nu}
This section follows \cite{CO, D, EA}. Throughout the paper, we will use the parameter $t$ instead of the parameter $t$ used in Introduction.
\subsection{General description}\label{ssec:S_nu_general}
%
For any $t \in \bC$, the category $\underline{Rep}(S_{t})$ is generated, as a $\bC$-linear Karoubian tensor category, by one object, denoted $\fh$. This object is the analogue of the permutation representation of $S_n$, and any object in $\underline{Rep}(S_{t})$ is a direct summand in a direct sum of tensor powers of $\fh$.

For $t \notin \bZ_{+}$, $\underline{Rep}(S_{t})$ is a semisimple abelian category.

\begin{notation}
 We will denote Deligne's category for integer value $n \geq 0$ of $t$ as $\underline{Rep}(S_{t=n})$, to distinguish it from the classical category $Rep(S_{n})$ of representations of the symmetric group $S_{n}$. Similarly for other categories arising in this text.
\end{notation}

If $t$ is a non-negative integer, then the category $\underline{Rep}(S_{t})$ has a tensor ideal $\idealI_{t}$, called the ideal of negligible morphisms (this is the ideal of morphisms $f: X \longrightarrow Y$ such that $tr(fu)=0$ for any morphism $u: Y \longrightarrow X$). In that case, the classical category $Rep(S_n)$ of finite-dimensional representations of the symmetric group for $n:=t$ is equivalent to $\underline{Rep}(S_{t=n})/\idealI_{t}$ (equivalent as Karoubian rigid symmetric monoidal categories).

The full, essentially surjective functor $\underline{Rep}(S_{t=n}) \rightarrow Rep(S_n)$ defining this equivalence will be denoted by $\mathcal{S}_n$.

Note that $\mathcal{S}_n$ sends $\fh$ to the permutation representation of $S_n$.
\begin{remark}
 Although $\underline{Rep}(S_{t})$ is not semisimple and not even abelian when $t \InnaA{=n} \in \bZ_+$, a weaker statement holds (see \cite[Proposition 5.1]{D}, Lemma \ref{lem:ss_subcat_Del}): consider the full subcategory $\underline{Rep}(S_{t \InnaA{=n}})^{( \InnaA{n}/2)}$ of $\underline{Rep}(S_{t})$ whose objects are directs summands of sums of $\fh^{\otimes m}, 0\leq m \leq  \InnaA{\frac{n}{2}}$. This subcategory is abelian semisimple, and the restriction $\InnaA{\mathcal{S}}_{n} \lvert_{\underline{Rep}(S_{t \InnaA{=n}})^{( \InnaA{n}/{2})}}$ is fully faithful.
\end{remark}

The indecomposable objects of $\underline{Rep}(S_{t})$, regardless of the value of $t$, are \InnaA{parametrized (up to isomorphism)} by all Young diagrams (of arbitrary size). We will denote the indecomposable object in $\underline{Rep}(S_{t})$ corresponding to the Young diagram $\T$ by $X_{\T}$.

For non-negative integer $t =:n$, we have: the partitions $\lambda$ for which $X_{\lambda}$ has a non-zero image in the quotient $\underline{Rep}(S_{t \InnaA{=n}})/\idealI_{t\InnaA{=n}} \cong Rep(S_n)$ are exactly the $\lambda$ for which $\lambda_1+\abs{\lambda} \leq n$.

If $\lambda_1+\abs{\lambda} \leq n$, then the image of $\lambda$ in $Rep(S_n)$ is the irreducible representation of $S_n$ corresponding to the Young diagram $\widetilde{\lambda}(n)$ (c.f. notation in Section \ref{sec:notation}).

\InnaA{This allows one to intuitively treat} the indecomposable objects of $\underline{Rep}(S_{t})$ as if they were parametrized by ``Young diagrams with a very long top row''. The indecomposable object $X_{\lambda}$ would be treated as if it corresponded to $\widetilde{\lambda}(t)$, i.e. a Young diagram obtained by adding a very long top row (``of size $t -\abs{\lambda}$''). This point of view is useful to understand how to extend constructions for $S_n$ involving Young diagrams to $\underline{Rep}(S_{t})$.

\begin{example}
 The indecomposable object $X_{\lambda}$, where $\lambda={\tiny \yng(6,5,4,1)}$ can be thought of as a Young diagram with a ``very long top row of length $(t- 16)$'':
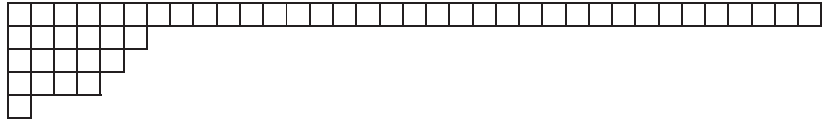
$$ \yng(35,6,5,4,1)$$
\end{example}

%
%

%
%
\subsection{Lifting objects}\label{ssec:nu_class_and_lift}
We start with an equivalence relation on the set of all Young diagrams, defined in \cite[Definition 5.1]{CO}:
\begin{definition}
Let $\lambda$ be any Young diagram, and set
$$\mu_{\lambda}(t) = (t -\abs{\lambda}, \lambda_1-1, \lambda_2-2, ...)$$
Given two Young diagrams $\lambda, \lambda'$, denote $\mu_{\lambda}(t) =: (\mu_0, \mu_1, ...), \mu_{\lambda'}(t) =: (\mu'_0, \mu'_1, ...)$.

We put $\lambda \stackrel{t}{\sim} \lambda'$ if there exists a bijection $f: \bZ_+ \rightarrow \bZ_+$ such that $\mu_i = \mu'_{f(i)}$ for any $i \geq 0$.

\end{definition}
We will call a $\stackrel{t}{\sim}$-class {\it trivial} if it contains exactly one Young diagram.

The following lemma is proved in \cite[Corollary 5.6, Proposition 5.8]{CO}:
\begin{lemma}\label{lem:nu_classes_struct}
\mbox{}
 \begin{enumerate}
 \item If $t \notin \bZ_+$, then any Young diagram $\lambda$ lies in a trivial $\stackrel{t}{\sim}$-class.
  \item The non-trivial $\stackrel{t}{\sim}$-classes are parametrized by all Young diagrams $\lambda$ such that $\widetilde{\lambda}(t)$ is a Young diagram (in particular, $t \in \bZ_+$), and are of the form $ \{\lambda^{(i)}\}_i$, with
\begin{align*}                                                                                                                                                                                                                                       &\lambda=\lambda^{(0)} \subset \lambda^{(1)} \subset \lambda^{(2)} \subset ... \\
&\text{ and } \lambda^{(i+1)} \setminus \lambda^{(i)} = \text{ strip in row } i+1 \text{ of length } \lambda_i -\lambda_{i+1} +1 \text{ for } i>0,\\
&\text{ and } \lambda^{(1)} \setminus \lambda^{(0)} = \text{ strip in row } 1 \text{ of length } t -\abs{\lambda} -\lambda_1 +1.                                                                                                                                                                                                                                    \end{align*}
 \end{enumerate}

\end{lemma}

We now consider Deligne's category $\underline{Rep}(S_T)$, where $T$ is a formal variable (c.f. \cite[Section 3.2]{CO}). This category is $\bC((T-t))$-linear, but otherwise it is very similar to Deligne's category $\underline{Rep}(S_{t})$ for generic $t$. For instance, as a $\bC((T-t))$-linear Karoubian tensor category, $\underline{Rep}(S_T)$ is generated by one object, again denoted by $\fh$.

One can show that $\underline{Rep}(S_T)$ is \InnaA{split} semisimple and thus abelian, and its simple objects are parametrized by Young diagrams of arbitrary size.

In \cite[Section 3.2]{CO}, Comes and Ostrik defined a map
 \begin{align*}
lift_{t}: \{\substack{\text{objects in } \underline{Rep}(S_{t}) \\ \text{up to isomorphism}}\}\rightarrow \{\substack{\text{objects in } \underline{Rep}(S_{T}) \\ \text{up to isomorphism}}\}
 \end{align*}

We will not give the precise definition of this map, but will list some of its useful properties. It is defined to be additive (i.e. $lift_{t}(A \oplus B) \cong lift_{t}(A) \oplus lift_{t}(B)$ for any $A, B \in \underline{Rep}(S_{t})$) and satisfies $lift_{t}(\fh)\cong \fh$. Moreover, we have:

\begin{proposition}\label{prop:lift_properties}
Let $A, B$ be two objects in $\underline{Rep}(S_{t})$.
 \begin{enumerate}
 \item $lift_{t}(A \otimes B) \cong lift_{t}(A) \otimes lift_{t}(B)$.
  \item $\dim_{\bC} \Hom_{\underline{Rep}(S_{t})} (A, B) = \dim_{Frak(\bC[[T]])} \Hom_{\underline{Rep}(S_{T})} (lift_{t}(A), lift_{t}(B))$.
  \item The map $lift_{t}$ is injective.
 \end{enumerate}
 \end{proposition}

 \begin{proof}
C.f. \cite[Proposition 3.12]{CO}.
 \end{proof}

\begin{remark}
It was proved both in \cite[Section 7.2]{D} and in \cite[Proposition 3.28]{CO} that the dimensions of the indecomposable objects $X_{\lambda}$ in $\underline{Rep}(S_{T})$ are polynomials in $T$ whose coefficients depend on $\lambda$ (given $\lambda$, this polynomial can be written down explicitly). Such polynomials are denoted by $P_{\lambda}(T)$.

Furthermore, it was proved in \cite[Proposition 5.12]{CO} that given $d \in \bZ_+$ and a Young diagram $\lambda$, $\lambda$ belongs to a trivial $\stackrel{d}{\sim}$-class iff $P_{\lambda}(d)=0$.
\end{remark}

The following result is proved in \cite[Lemma 5.20]{CO}:

\begin{lemma}[Comes, Ostrik]\label{lem:nu_classes_lift}
Consider the $\stackrel{t}{\sim}$-equivalence relation on Young diagrams.

\begin{itemize}
 \item Whenever $\lambda$ lies in a trivial $\stackrel{t}{\sim}$-class, $lift_{t}(X_{\lambda}) = X_{\lambda}$.
\item For a non-trivial $\stackrel{t}{\sim}$-class $\{ \lambda^{(i)} \}_i$,
\begin{align*}
& lift_{t}(X_{\lambda^{(0)}}) = X_{\lambda^{(0)}},
& lift_{t}(X_{\lambda^{(i)}}) = X_{\lambda^{(i)}} \oplus X_{\lambda^{(i-1)}} \; \, \forall i \geq 1
\end{align*}
\end{itemize}
\end{lemma}

Based on Lemmas \ref{lem:nu_classes_struct}, \ref{lem:nu_classes_lift}, Comes and Ostrik prove the following theorem (c.f. \cite[Theorem 5.3, Proposition 5.22, Theorems 6.4, 6.10]{CO}, \cite[Proposition 2.7]{CO2}):

\begin{theorem}\label{thrm:blocks_S_nu}
 The indecomposable objects $X_{\lambda}, X_{\lambda'}$ belong to the same block of $\underline{Rep}(S_{t})$ iff $\lambda \stackrel{t}{\sim} \lambda'$. The structure of the blocks of $\underline{Rep}(S_{t})$ is described below:
 \begin{itemize}[leftmargin=*]
  \item For a trivial $\stackrel{t}{\sim}$-class $\{ \lambda \}$, the object $X_{\lambda}$ satisfies:
  $$\dim \End_{\underline{Rep}(S_{t})} (X_{\lambda}) =1$$
  and the block of $\underline{Rep}(S_{t})$ corresponding to $\{\lambda\}$ is equivalent to the category $Vect_{\bC}$ of finite dimensional complex vector spaces (in particular, it is a semisimple abelian category, so we will call these blocks {\it semisimple}).
  \item Let $ \{\lambda^{(i)}\}_i$ be a non-trivial $\stackrel{t}{\sim}$-class, and let $i \geq 1, j \geq 0$. Then the block corresponding to $ \{\lambda^{(i)}\}_i$ is not an abelian category (in particular, not semisimple), and the objects $X_{\lambda^{(i)}}$ satisfy:
\begin{align*}
&\dim \Hom_{\underline{Rep}(S_{t})} \left( X_{\lambda^{(j)}}, X_{\lambda^{(i)}} \right) = 0 \text{ if } \InnaA{\abs{j -i} \geq 2}\\
&\dim \Hom_{\underline{Rep}(S_{t})} \left( X_{\lambda^{(j)}}, X_{\lambda^{(i)}} \right) = 1 \text{ if } \InnaA{\abs{j -i} =1}\\
&\dim \End_{\underline{Rep}(S_{t})} \left( X_{\lambda^{(i)}} \right) = 2 \text{ for } i\geq 1 \\
& \dim \End_{\underline{Rep}(S_{t})} \left( X_{\lambda^{(0)}} \right) = 1
\end{align*}

This block has the following associated quiver:
$$ X_{\lambda^{(0)}} \substack{ \alpha_0 \\ \leftrightarrows \\ \beta_0} X_{\lambda^{(1)}}  \substack{ \alpha_1 \\ \leftrightarrows \\ \beta_1} X_{\lambda^{(2)}}  \substack{ \alpha_2 \\ \leftrightarrows \\ \beta_2} ...$$
with relations \InnaA{$ \alpha_0 \circ \beta_0 =0$}, $\beta_i \circ \beta_{i-1} =0$, \InnaA{$\alpha_i \circ \alpha_{i-1} = 0$, $\beta_i \circ \alpha_i = \alpha_{i+1} \circ \beta_{i+1}$ for $i \geq 0$}.
 \end{itemize}

\end{theorem}

\subsection{Semisimple subcategory}
 \begin{lemma}\label{lem:ss_subcat_Del}
  Let $N \in \bZ_+$. The full Karoubian subcategory $\underline{Rep}(S_t)^{(N)}$ of $\underline{Rep}(S_t)$ generated by the objects $\fh^{\otimes r}$, $r \in \{0, ..., N\}$ is semisimple when $t \notin \{0, ..., 2N-2\}$.
 \end{lemma}
 \begin{proof}
 Direct consequence of Lemma \ref{lem:nu_classes_lift} (c.f. also \cite[Proposition 5.1]{D}).
 \end{proof}
 \begin{remark}
  Note that the simple objects in this category are exactly the indecomposable objects $X_{\lambda}$ for which $\abs{\lambda} \leq N$.
 \end{remark}
\subsection{Abelian envelope}\label{ssec:S_nu_abelian_env}

As it was mentioned before, the category $\underline{Rep}(S_{t})$ is defined as a Karoubian category. For $t \notin \bZ_+$, it is semisimple and thus abelian, but for $t \in \bZ_+$, it is not abelian. Fortunately, it has been shown that $\underline{Rep}(S_{t})$ possesses an ``abelian envelope'', that is, that it can be embedded into an abelian tensor category, and this abelian tensor category has a universal mapping property (c.f. \cite[Conjecture 8.21.2]{D}, and \cite[Theorem 1.2]{CO2}).

An explicit construction of the category $\underline{Rep}^{ab}(S_{t=n})$ is given in \cite{CO2}. We will only list the results which will be used in this paper.

\begin{proposition}\label{prop:ab_envelope_highest_weight}
 The category $\underline{Rep}^{ab}(S_{t})$ is a highest weight category corresponding to the (infinite) partially ordered set $(\{ \text{Young diagrams} \}, \geq )$, where
 $$\lambda \geq \mu \text{ iff } \lambda \stackrel{t}{\sim} \mu, \lambda \subset \mu$$
 (namely, $\lambda^{(i)} \geq \lambda^{(j)}$ if $i \leq j$).

\end{proposition}

The category $\underline{Rep}(S_t)$ is the subcategory of tilting objects in $\Dab$. More specifically, we have:
\begin{proposition}\label{prop:obj_ab_env}
The blocks of the abelian category $\Dab$ correspond to the $\stackrel{t}{\sim}$-classes:
\begin{enumerate}
   \item Let $\lambda$ lie in a trivial $\stackrel{t}{\sim}$-class. The corresponding block of $\Dab$ is semisimple, and is generated by the simple projective object $X_{\lambda}$.
   \item Let $ \{\lambda^{(i)}\}_{\InnaA{i \geq 0}}$ be a non-trivial $\stackrel{t}{\sim}$-class. The corresponding block of $\Dab$ is non-semisimple, and contains $X_{\lambda^{(i)}}$ for $i \geq 0$.

 \begin{itemize}[leftmargin=*]
 \item $X_{\lambda^{(0)}}$ is a simple, non-projective object.
  \item For any $i \geq 1$, $X_{\lambda^{(i)}}$ is a projective object.
%
 \end{itemize}
\end{enumerate}

\end{proposition}

\section{Reduced Kronecker coefficients and Deligne's categories}\label{sec:red_Kron_coeff_and_Deligne_cat}

\begin{notation}
 Let $t \in \bC$, and consider Deligne's category $\underline{Rep}(S_t)$. Let $A$ be any object in $\underline{Rep}(S_t)$, and $X_{\lambda}$ be an indecomposable object. Then $A$ decomposes into a direct sum of indecomposable objects, and we denote by $[A: X_{\lambda}]_t$ the multiplicity of $X_{\lambda}$ in this direct sum, i.e.
 $$ A \cong \bigoplus_{\lambda} [A: X_{\lambda}]_t X_{\lambda}$$
 
 Similarly, given an object $A$ in the $Frak(\bC[[T]])$-linear Deligne's category $\underline{Rep}(S_T)$, and an indecomposable object $X_{\lambda}$ in the same category, we denote by $[A: X_{\lambda}]_{T}$ the multiplicity of $X_{\lambda}$ in the decomposition of $A$ into a direct sum of indecomposable objects. Since $\underline{Rep}(S_T)$ is semisimple, we have:
 $$[A: X_{\lambda}]_{T} = \dim_{Frak(\bC[[T]])} \Hom_{\underline{Rep}(S_T)}(A,  X_{\lambda})$$
\end{notation}

\begin{definition}
 Consider the $Frak(\bC[[T]])$-linear Deligne's category $\underline{Rep}(S_T)$ (this category is semisimple). Let $\lambda, \mu, \T$ be Young diagrams of arbitrary size. We denote by $\bar{g}^{\lambda}_{\mu, \tau}$ the multiplicity of the simple object $X_{\lambda}$ in $X_{\mu} \otimes X_{\tau}$:
 $$ \bar{g}^{\lambda}_{\mu, \tau} := [ X_{\mu} \otimes X_{\tau}: X_{\lambda}]_{T} = \dim_{Frak(\bC[[T]])} \Hom_{\underline{Rep}(S_T)}(X_{\mu} \otimes X_{\tau}, X_{\lambda})$$
 The value $\bar{g}^{\lambda}_{\mu, \tau}$ will be called {\it the reduced Kronecker coefficient} corresponding to the triple of Young diagrams $(\lambda, \mu, \T)$.
\end{definition}

Thus the reduced Kronecker coefficients are the structural constants of the Grothendieck rings $\mathfrak{R}(S_t)$ of Deligne's categories $\underline{Rep}(S_t)$ at generic values of $t$ ($t \notin \bZ_+$). These rings are all isomorphic to one another and do not depend on $t$:
\begin{proposition}\label{prop:groth_ring_generic_t}
 Let $t \notin \bZ_+$, and let $\lambda, \mu, \T$ be Young diagrams of arbitrary size. Consider the semisimple category $\underline{Rep}(S_t)$. Then the multiplicity of the simple object $X_{\lambda}$ in $X_{\mu} \otimes X_{\tau}$ is $\bar{g}^{\lambda}_{\mu, \tau}$.
\end{proposition}
\begin{proof}
 By Lemma \ref{lem:nu_classes_lift}, $lift_t(X_{\lambda}) \cong X_{\lambda}$ for any Young diagram $\lambda$, and so $lift_t(X_{\mu} \otimes X_{\tau}) \cong X_{\mu} \otimes X_{\tau}$. 
 Thus we get:
 \begin{align*}
 &\bar{g}^{\lambda}_{\mu, \tau} := [ X_{\mu} \otimes X_{\tau}: X_{\lambda}]_{T} = \dim_{Frak(\bC[[T]])} \Hom_{\underline{Rep}(S_T)}(X_{\mu} \otimes X_{\tau}, X_{\lambda}) = \\
&= \dim_{\bC} \Hom_{\underline{Rep}(S_t)}(X_{\mu} \otimes X_{\tau}, X_{\lambda}) = [ X_{\mu} \otimes X_{\tau}: X_{\lambda}]_{t}.
 \end{align*}

\end{proof}

In fact, for a fixed triple $(\lambda, \mu, \T)$, the same is true for almost all values of $t$:

\begin{proposition}\label{prop:red_kron_almost_all_mult}
Fix $\lambda, \mu, \T$, and let $N: = \max \{\abs{\lambda}, \abs{\mu} + \abs{\T} \}$. Let $t \notin \{0, ..., 2N-2 \}$. Consider the category $\underline{Rep}(S_t)$. Then the multiplicity of the simple object $X_{\lambda}$ in $X_{\mu} \otimes X_{\tau}$ is $\bar{g}^{\lambda}_{\mu, \tau}$.
\end{proposition}
\begin{proof}
 Almost the same arguments as in Proposition \ref{prop:groth_ring_generic_t} apply here: 
 
  By Lemma \ref{lem:nu_classes_lift}, $lift_t(X_{\lambda}) \cong X_{\lambda}$ and $$lift_t(X_{\mu} \otimes X_{\tau}) \cong lift_t(X_{\mu}) \otimes lift_t(X_{\T}) \cong X_{\mu} \otimes X_{\tau}$$
 Again,
 \begin{align*}
 &\bar{g}^{\lambda}_{\mu, \tau} =\dim_{Frak(\bC[[T]])} \Hom_{\underline{Rep}(S_T)}(X_{\mu} \otimes X_{\tau}, X_{\lambda}) = \\
&= \dim_{\bC} \Hom_{\underline{Rep}(S_t)}(X_{\mu} \otimes X_{\tau}, X_{\lambda})
 \end{align*}

The objects $X_{\mu} \otimes X_{\tau}$, $ X_{\lambda}$ are direct summands in $\fh^{\otimes \abs{\mu} + \abs{\T}}$, $\fh^{\otimes \abs{\lambda}}$ respectively. Therefore, they lie in $\underline{Rep}(S_t)^{(N)}$, which is semisimple (c.f. Lemma \ref{lem:ss_subcat_Del}). 

We conclude that $$\bar{g}^{\lambda}_{\mu, \tau} = \dim_{\bC} \Hom_{\underline{Rep}(S_t)}(X_{\mu} \otimes X_{\tau}, X_{\lambda}) = [ X_{\mu} \otimes X_{\tau}: X_{\lambda}]_{t}. $$
\end{proof}

To conclude this section, we prove a lemma which will be useful later on:
\begin{lemma}\label{lem:lifting_multip}
Let $\mu, \tau, \lambda$ be three Young diagrams, and let $n \in \bZ, n \geq \abs{\lam} + \lam_1$. Denote by $\{ \lambda^{(i)} \}_{i \geq 0}$ the $\stackrel{n}{\sim}$-class of $\lambda$ ($\lambda = \lambda^{(0)}$ since  $n \geq \abs{\lambda} +\lambda_1$).
Then 
$$ [X_{\mu} \otimes X_{\T}: X_{\lam}]_{t=n} = \sum_{j \geq 0} (-1)^j \dim_{Frak(\bC[[T]])} \Hom_{\underline{Rep}(S_T)}(lift_{t=n}(X_{\mu}) \otimes lift_{t=n}(X_{\tau}), X_{\lambda^{(j)}}) $$
\end{lemma}
\begin{proof}
By definition, we have:
$$ X_{\mu} \otimes X_{\T} = \bigoplus_{\rho \text{ a Young diagram}} [X_{\mu} \otimes X_{\T}: X_{\rho}]_{t=n} X_{\rho}$$
so $$ lift_{t=n}(X_{\mu} \otimes X_{\T}) = \bigoplus_{\rho \text{ a Young diagram}} [X_{\mu} \otimes X_{\T}: X_{\rho}]_{t=n} lift_{t=n}(X_{\rho})$$
On the other hand, 
$$ lift_{t=n}(X_{\mu} \otimes X_{\T}) = \bigoplus_{\rho \text{ a Young diagram}} \dim_{Frak(\bC[[T]])} \Hom_{\underline{Rep}(S_T)}(lift_{t=n}(X_{\mu}) \otimes lift_{t=n}(X_{\tau}), X_{\rho}) X_{\rho}$$
Now, by Lemma \ref{lem:nu_classes_lift}, $$lift_{t=n}(X_{\lam}) = X_{\lam}, \; \; \; lift_{t=n}(X_{\lam^{(i)}}) = X_{\lam^{(i)}} \oplus X_{\lam^{(i-1)}} \text{ for } i \geq 1$$
so for any $i \geq0$, $$\dim_{Frak(\bC[[T]])} \Hom_{\underline{Rep}(S_T)}(lift_{t=n}(X_{\mu}) \otimes lift_{t=n}(X_{\tau}), X_{\lambda^{(i)}}) = [X_{\mu} \otimes X_{\T}: X_{\lam^{(i)}}]_{t=n} +  [X_{\mu} \otimes X_{\T}: X_{\lam^{(i+1)}}]_{t=n}$$
and thus 
$$ [X_{\mu} \otimes X_{\T}: X_{\lam^{(i)}}]_{t=n} = \sum_{j \geq 0} (-1)^j \dim_{Frak(\bC[[T]])} \Hom_{\underline{Rep}(S_T)}(lift_{t=n}(X_{\mu}) \otimes lift_{t=n}(X_{\tau}), X_{\lambda^{(i+j)}}) $$
The statement of the lemma is just the special case when $i=0$.
\end{proof}

\section{Properties of reduced Kronecker coefficients}\label{sec:prop_Kron_coeff}
\subsection{Symmetry}\label{ssec:sym_Kron_coeff}
 The reduced Kronecker coefficient $\bar{g}^{\lambda}_{\mu, \tau}$ is symmetric in terms of the three partitions $\lambda, \mu, \T$.

In the context of the Deligne category $\underline{Rep}(S_{t})$ for a generic $t \in \bC$, this corresponds to the fact that $$\bar{g}^{\lambda}_{\mu, \tau} = \dim_{\bC} \Hom_{S_{t}}(X_{\lambda}, X_{\mu} \otimes X_{\tau})$$
and $$\Hom_{S_{t}}(X_{\lambda}, X_{\mu} \otimes X_{\tau}) \cong \Hom_{S_{t}}(1, X_{\lambda} \otimes X_{\mu} \otimes X_{\tau}) $$
(since any object in $\underline{Rep}(S_{t})$ is self-dual). The last expression is clearly symmetric in $\lambda, \mu, \T$.
\subsection{Murnaghan-Littlewood inequalities}\label{ssec:Murn_Littlewood_ineq}
We now give a sufficient condition on the Young diagrams for the reduced Kronecker coefficient to be zero. This is called the Murnaghan-Littlewood inequalities.

\begin{lemma}
Let $\lambda, \mu, \T$ be three partitions of arbitrary sizes, and let $t \in \bC$.

Then $[X_{\mu} \otimes X_{\tau}: X_{\lambda}]_t = 0$ whenever $\abs{\lambda} > \abs{\mu} + \abs{\T}$.
\end{lemma}
\begin{proof}
 Recall from the construction of Deligne's category (\cite[Section 2]{CO}) that the indecomposable object $X_{\rho}$ appears as a direct summand of $\fh^{\otimes \abs{\rho}}$, but does not appear in smaller tensor powers of $\fh$. Thus $X_{\mu} \otimes X_{\tau}$ is a direct summand of $\fh^{\otimes (\abs{\mu} + \abs{\T})}$, while $X_{\lambda}$ cannot appear as a direct summand in $\fh^{\otimes (\abs{\mu} + \abs{\T})}$ if $\abs{\lambda} > \abs{\mu} + \abs{\T}$.
\end{proof}

Applying this lemma to the case $t \notin \bZ_+$, we immediately obtain the following well-known Murnaghan-Littlewood inequalities (c.f. \cite{Mu1}):

\begin{corollary}\label{cor:red_Kron_zero}
Let $\lambda, \mu, \T$ be three partitions of arbitrary sizes.
Then $\bar{g}^{\lambda}_{\mu, \tau} \neq 0$ implies $\abs{\lambda} \leq \abs{\mu} + \abs{\T}$.
\end{corollary}

Of course, since $\bar{g}^{\lambda}_{\mu, \tau}$ is symmetric with respect to the three Young diagrams, we conclude that $\abs{\T} \leq \abs{\mu} +\abs{\lambda}, \abs{\mu} \leq \abs{\lambda} +\abs{\T}$ as well.
\subsection{Reduced Kronecker coefficients and Littlewood-Richardson coefficients}\label{ssec:Kron_vs_LR_coeff}
The following proposition is proved in \cite[Proposition 5.11]{D}:
\begin{proposition}
When $\abs{\lambda} = \abs{\mu} +\abs{\T}$, the reduced Kronecker coefficient $\bar{g}^{\lambda}_{\mu, \tau}$ is equal to the Littlewood-Richardson coefficient $$c^{\lambda}_{\mu, \tau} := \dim_{\bC} \Hom_{S_{\abs{\mu}} \times S_{\abs{\tau}}} (Res_{S_{\abs{\mu}} \times S_{\abs{\tau}}}^{S_{\abs{\lambda}}} \lambda, \mu \otimes \tau)$$ 

\end{proposition}
We give a sketch of the proof following \cite[Sections 2, 5]{D}.
\begin{proof}[Sketch of proof]
The construction of any indecomposable object $X_{\lambda}$ in $\underline{Rep}(S_t)$ can be done in two ways: one is to consider $X_{\lambda}$ as a direct summand of $\fh^{\otimes \abs{\lambda}}$, and the other is to consider $X_{\lambda}$ as a direct summand of $\Delta_{\abs{\lambda}}$. 

Let $k = \abs{\lambda}$. Consider the action of $S_{k}$ on $\Delta_{k}^*$, which is the largest direct summand of $\Delta_{k}$ having no common direct summands with $\Delta_{k-1}$. The decomposition of $\Delta_{k}^*$ into irreducible $S_{k}$-representations with respect to this action gives us 
$$\Delta_{k}^* = \bigoplus_{\abs{\rho} = k} X_{\rho} \otimes \rho$$

Now, $X_{\mu} \otimes X_{\tau}$ is a direct summand of $\Delta_{\abs{\mu}} \otimes \Delta_{\abs{\T}}$. The latter decomposes as a direct sum of $\Delta_k$ for $k \leq \abs{\mu} +\abs{\T} = \abs{\lambda}$, with $\Delta_{\abs{\lambda}}$ appearing with multiplicity $1$. This allows us to conclude that
$$[X_{\mu} \otimes X_{\tau}: X_{\lambda}]=\dim_{\bC} \Hom_{S_{\abs{\mu}} \times S_{\abs{\tau}}} (Res_{S_{\abs{\mu}} \times S_{\abs{\tau}}}^{S_{\abs{\lambda}}} \lambda, \mu \otimes \tau) =: c^{\lambda}_{\mu, \tau}.$$

\end{proof}
\begin{corollary}
For any $t \notin \bZ_+$, fix a filtration on the Grothendieck ring $\mathfrak{R}(S_t)$ of $\underline{Rep}(S_t)$ by setting $X_{\lambda} \in Filtra_{\abs{\lambda}}(\mathfrak{R}(S_t))$. Then the associated graded ring $Gr(\mathfrak{R}(S_t))$ is isomorphic to the ring of symmetric functions, as defined in \cite[Chapter I, Par. 5]{Mac}, with $X_{\lambda}$ corresponding to the Schur function $s_{\lambda}$.
\end{corollary}

\begin{remark}
The ring $Gr(\mathfrak{R}(S_t))$ is also isomorphic to the Grothendieck ring of the tensor category $\bigoplus_{n \geq 0} Rep(S_n)$, with the tensor product given by the Bernstein-Zelevinsky product $$\mu \otimes \tau := Ind_{S_{\abs{\mu}} \times S_{\abs{\tau}}}^{S_{\abs{\mu} +\abs{\T}}} \mu \otimes \tau$$ 

One can show that the same is true for the all the abelian categories $\underline{Rep}^{ab}(S_{t})$: taking an appropriate filtration on the Grothendieck ring of $\underline{Rep}^{ab}(S_{t})$, the associated graded ring will be isomorphic to the ring of symmetric functions. The filtration should be taken so that the simple object corresponding to the Young diagram $\lam$ lies in the filtra $\abs{\lam}$.
\end{remark}

\subsection{Reduced Kronecker coefficients and standard Kronecker coefficients}\label{ssec:red_Kron_vs_Kron}

The following proposition shows that the reduced Kronecker coefficient $\bar{g}^{\lambda}_{\mu, \tau}$ is a stable value of a sequence of standard Kronecker coefficients, and gives a formula for recovering standard Kronecker coefficients from reduced Kronecker coefficients. This formula appears in \cite[Theorem 1.1]{BOR} in a slightly different form; we show that it can be obtained directly from object lifting for Deligne categories. 

\begin{proposition}\label{prop:Kron_vs_reduced_Kron}
 Let $\lambda, \mu, \T$ be three partitions of arbitrary sizes, and let 
 
 $N:= \max \{ \abs{\lambda} + \lambda_1, \abs{\mu} + \mu_1, \abs{\T} + \T_1 \}$. 
 \begin{enumerate}
  \item Consider the sequence $$\{ g^{\widetilde{\lambda}(n)}_{\widetilde{\mu}(n), \widetilde{\tau}(n)}\}_{n \geq N} $$ of standard Kronecker coefficients. This sequence stabilizes, and the stable value is the reduced Kronecker coefficient $\bar{g}^{\lambda}_{\mu, \tau}$.
  \item Let $n \geq N$. Denote by $\{ \lambda^{(i)} \}_{i \geq 0}$ the $\stackrel{n}{\sim}$-class of $\lambda$ ($\lambda = \lambda^{(0)}$ since  $n \geq \abs{\lambda} +\lambda_1$). Then $$g^{\widetilde{\lambda}(n)}_{\widetilde{\mu}(n), \widetilde{\tau}(n)} = [X_{\mu} \otimes X_{\tau}: X_{\lambda = \lambda^{(0)}}]_{t=n} = \sum_{i \geq 0} (-1)^{i} \bar{g}^{\lambda^{(i)}}_{\mu, \tau} $$
  
  \item The stable value $\bar{g}^{\lambda}_{\mu, \tau}$ is the maximum of the sequence $\{ g^{\widetilde{\lambda}(n)}_{\widetilde{\mu}(n), \widetilde{\tau}(n)}\}_{n \geq N} $. Namely, $$ \bar{g}^{\lambda}_{\mu, \tau} \geq g^{\widetilde{\lambda}(n)}_{\widetilde{\mu}(n), \widetilde{\tau}(n)} $$
 \end{enumerate}

\end{proposition}
\begin{remark}
 Note that Corollary \ref{cor:red_Kron_zero} implies that the sum $\sum_{i \geq 0} (-1)^{i} \bar{g}^{\lambda^{(i)}}_{\mu, \tau} $ is finite (due to the fact that $\{\abs{\lambda^{(i)}} \}_i$ is a strongly increasing sequence).
\end{remark}

\begin{proof}
Let $n \geq N$. Recall the symmetric monoidal functor $\mathcal{S}_n: \underline{Rep}(S_{t=n}) \rightarrow Rep(S_n)$ described in Section \ref{sec:Del_cat_S_nu}. Due to the requirement on $N$, we have: 
$$\mathcal{S}_n(X_{\lambda}) \cong \widetilde{\lambda}(n), \, \mathcal{S}_n(X_{\mu}) \cong \widetilde{\mu}(n), \, \mathcal{S}_n(X_{\T}) \cong \widetilde{\tau}(n)$$
and since $\mathcal{S}_n$ preserves tensor products, we see that $$[ X_{\mu} \otimes X_{\tau}: X_{\lambda}]_{t = n} = g^{\widetilde{\lambda}(n)}_{\widetilde{\mu}(n), \widetilde{\tau}(n)}$$
\begin{enumerate}
\item Let $n \geq 2(\abs{\mu}+\abs{\tau})$. By Proposition \ref{prop:red_kron_almost_all_mult}, 
$$\bar{g}^{\lambda}_{\mu, \tau} = [ X_{\mu} \otimes X_{\tau}: X_{\lambda}]_{t = n} $$

and thus $$ \bar{g}^{\lambda}_{\mu, \tau} = g^{\widetilde{\lambda}(n)}_{\widetilde{\mu}(n), \widetilde{\tau}(n)}$$ for any $n \geq 2(\abs{\mu}+\abs{\tau})$ (in fact, one can use Lemma \ref{lem:nu_classes_lift} to show that this stable value is reached when $n \geq \abs{\mu}+\abs{\tau}+\mu_1+\tau_1$). This proves Part (1).

\item Let $n \geq N$.
The objects $X_{\lambda}, X_{\mu}, X_{\T}$ are all minimal in their respective $\stackrel{n}{\sim}$-classes, so by Lemma \ref{lem:nu_classes_lift}, $$lift_{t=n}(X_{\mu}) \cong X_{\mu}, \, lift_{t=n}(X_{\T}) \cong X_{\T}$$
By Lemma \ref{lem:lifting_multip}, we have:
\begin{align*}
 &[X_{\mu} \otimes X_{\T}: X_{\lambda = \lambda^{(0)}}]_{t=n} = \sum_{i \geq 0} (-1)^i \dim_{Frak(\bC[[T]])} \Hom_{\underline{Rep}(S_T)}(lift_{t=n}(X_{\mu}) \otimes lift_{t=n}(X_{\tau}), X_{\lambda^{(i)}}) = \\
 &=\sum_{i \geq 0} (-1)^i [X_{\mu} \otimes X_{\tau}: X_{\lambda^{(i)}}]_T = \sum_{i \geq 0} (-1)^{i} \bar{g}^{\lambda^{(i)}}_{\mu, \tau} 
\end{align*}
and so
$$g^{\widetilde{\lambda}(n)}_{\widetilde{\mu}(n), \widetilde{\tau}(n)} = [X_{\mu} \otimes X_{\tau}: X_{\lambda}]_{t=n} = \sum_{i \geq 0} (-1)^{i} \bar{g}^{\lambda^{(i)}}_{\mu, \tau} $$
which completes the proof of Part (2).

\item We proved in the proof of Lemma \ref{lem:lifting_multip} that $$ [X_{\mu} \otimes X_{\T}: X_{\lambda^{(1)}}]_{t=n} = \sum_{i \geq 1} (-1)^{i-1} [X_{\mu} \otimes X_{\tau}: X_{\lambda^{(i)}}]_T = \sum_{i \geq 1} (-1)^{i-1} \bar{g}^{\lambda^{(i)}}_{\mu, \tau} $$

This multiplicity is a non-negative number, and we just showed that $$\sum_{i \geq 1} (-1)^{i-1} \bar{g}^{\lambda^{(i)}}_{\mu, \tau}  = \bar{g}^{\lambda^{(0)}}_{\mu, \tau} - g^{\widetilde{\lambda}(n)}_{\widetilde{\mu}(n), \widetilde{\tau}(n)}$$ So $$\bar{g}^{\lambda^{(0)} = \lambda}_{\mu, \tau} \geq g^{\widetilde{\lambda}(n)}_{\widetilde{\mu}(n), \widetilde{\tau}(n)}.$$
\end{enumerate}
\end{proof}

\begin{remark}
 The fact that the sequence $\{ g^{\widetilde{\lambda}(n)}_{\widetilde{\mu}(n), \widetilde{\tau}(n)} \}_n$ stabilizes was proved by Murnaghan, see \cite{Mu1, Mu2}. The reduced Kronecker coefficients (sometimes also called ``stable Kronecker coefficients'') were originally defined as the stabilizing values of such sequences (see \cite{BOR}, for example).
 
 In addition, it was proved in \cite{Br} that this is a weakly increasing sequence (in particular, this implies that its stable valu is its maximum).
\end{remark}
\begin{remark}
 Similarly to the proof of Proposition \ref{prop:Kron_vs_reduced_Kron}, Part (3), one can prove the following statement: consider the partial sums $$P_k := \sum_{0 \leq i \leq k} (-1)^i \bar{g}^{\lambda^{(i)}}_{\mu, \tau}$$ Then for any $k\geq 0$, $$P_{2k} \leq g^{\widetilde{\lambda}(n)}_{\widetilde{\mu}(n), \widetilde{\tau}(n)} \leq P_{2k+1}$$ 
 (the above proposition tells us that $g^{\widetilde{\lambda}(n)}_{\widetilde{\mu}(n), \widetilde{\tau}(n)}$ is the stable value of the sequence $\{P_k \}_{k \geq 0}$).
\end{remark}

$$ $$
Let $u = (u_1,u_2, ...)$ be a sequence of integers. Denote $$u^{\dag i}:= (u_1+1, ..., u_{i-1}+1, u_{i+1}, ...)$$
(this is the sequence obtained from $u$ by removing the $i$-th term and adding $1$ to all the previous terms). Of course, if $u$ was a weakly decreasing sequence (Young diagram), so is $u^{\dag i}$.

Also, given a Young diagram $\lambda$, denote by $\bar{\lambda}$ the Young diagram obtained from $\lambda$ by removing the top row (thus given $\lambda\vdash n$, $\mathcal{S}_n (X_{\bar{\lambda}}) = \lambda$). 
With these definitions, the second statement of Proposition \ref{prop:Kron_vs_reduced_Kron} can also be reformulated as follows (in this form it appears in \cite{BOR}):
\begin{proposition}
Let $n \in \bZ_+$, and let $\lambda, \mu, \T$ be three partitions of $n$.
Then 
$$g^{\lambda}_{\mu, \T} =  \sum_{i \geq 1} (-1)^{i+1} \bar{g}^{\lambda^{\dag i}}_{\bar{\mu}, \bar{\tau}} $$
\end{proposition}

\begin{proof}
 First of all, note that $\bar{\lambda}$ satisfies: $$\abs{\bar{\lambda}} + \bar{\lambda}_1 = n - \lambda_1 +\lambda_2 \leq n$$ 
 Thus $\bar{\lambda}$ is the minimal element in a non-trivial $\stackrel{n}{\sim}$-class. All we need to do is check that $\{\lambda^{\dag i}\}_{i \geq 1}$ is exactly the $\stackrel{n}{\sim}$-class of $\bar{\lambda}$.
 
 Denote by $\{\bar{\lambda}^{(i)}\}_{i \geq 0}$ the $\stackrel{n}{\sim}$-class of $\bar{\lambda}$. Of course, $$\bar{\lambda}^{(0)} = \bar{\lam} = \lam^{\dag 1}$$
 
 Now, by Lemma \ref{lem:nu_classes_struct}, we have:
 \begin{align*}
  &\bar{\lambda}^{(1)} = (\bar{\lam}_1+ n -\abs{\bar{\lambda}} -\bar{\lambda}_1 +1, \bar{\lam}_2, \bar{\lam}_3, ...) = \\
  &=( n - (n-\lam_1) +1, \bar{\lam}_2, \bar{\lam}_3, ...) = (\lam_1+1, \lam_3, \lam_4, ...) = \lam^{\dag 2}
 \end{align*}
 and proceeding by induction on $i$, we have (for $i>1$):
 \begin{align*}
  &\bar{\lambda}^{(i)} = (\bar{\lam}^{(i-1)}_1, \bar{\lam}^{(i-1)}_2, ..., \bar{\lam}^{(i-1)}_{i-1}, \bar{\lam}^{(i-1)}_i - \bar{\lam}_i + \bar{\lam}_{i-1}+1, \bar{\lam}^{(i-1)}_{i+1},...) =\\
  & = (\lam_1+1, \lam_2+1, ..., \lam_{i-2}+1, \lam_{i-1}+1, \lam_{i+1} -\bar{\lam}_i + \bar{\lam}_{i-1}+1, \lam_{i+2}, ...) = \lam^{\dag (i+1)}
  \end{align*}
  Thus we proved that the sequences $\{\lambda^{\dag i}\}_{i \geq 1}$, $\{\bar{\lambda}^{(i)}\}_{i \geq 0}$ of Young diagrams coincide, as wanted.
 \end{proof}

\subsection{New formulas involving reduced Kronecker coefficients}\label{ssec:new_form_Kron}

In this subsection we present close companions of the formula in Proposition \ref{prop:Kron_vs_reduced_Kron}, which are based on the following standard property of rigid symmetric monoidal abelian categories: given an object $X$ and a projective object $P$, the tensor product $X \otimes P$ is again projective. Note that the second part of Proposition \ref{prop:alt_sum_reduced_Kron} is just a generalization of Proposition \ref{prop:Kron_vs_reduced_Kron}, Part (2).

\begin{proposition}\label{prop:alt_sum_reduced_Kron}
 Let $\lam, \mu, \T$ be three Young diagrams. Let $n \in \bZ$, $ \abs{\lam} + \lam_1 \leq n$. 
 
 Denote by $\{ \lambda^{(i)} \}_{i \geq 0}$ the $\stackrel{n}{\sim}$-class of $\lambda$ ($\lambda = \lambda^{(0)}$ since  $n \geq \abs{\lambda} +\lambda_1$). 
 \begin{itemize}
  \item If $X_{\mu}$ lies in a trivial $\stackrel{n}{\sim}$-class (equivalently, if $n \in \{ \abs{\mu} +\mu_l -l : \, l= 1, ..., \abs{\mu}\}$), then 
  $$ \sum_{i \geq 0} (-1)^{i} \bar{g}^{\lambda^{(i)}}_{\mu, \tau} =[X_{\mu} \otimes X_{\tau}: X_{\lambda = \lambda^{(0)}}]_{t=n} = 0$$
  \item Assume both $X_{\mu}$ and $X_{\T}$ lie in non-trivial $\stackrel{n}{\sim}$-classes, denoted by $\{\mu^{(i)} \}_{i \geq 0}$ and $\{\T^{(i)} \}_{i \geq 0}$ respectively (these classes are not necessarily distinct). Let $k, l$ be such that $\mu = \mu^{(k)}$, $\T = \T^{(l)}$. Then 
  $$ \sum_{i \geq 0} (-1)^{i} \bar{g}^{\lambda^{(i)}}_{\mu, \tau} = (-1)^{k+l} g^{\widetilde{\lambda}(n)}_{\widetilde{\mu^{(0)}}(n), \widetilde{\tau^{(0)}}(n)}$$
 \end{itemize}
\end{proposition}
\begin{remark}
 Again, by Corollary \ref{cor:red_Kron_zero} the above sums are finite (due to the fact that $\{\abs{\lambda^{(i)}} \}_i$ is a strongly increasing sequence).
\end{remark}
\begin{proof}
First of all, recall that the case when $n \geq \abs{\mu} + \mu_1, \abs{\T} + \T_1$ has been studied in Proposition \ref{prop:Kron_vs_reduced_Kron}. So without loss of generality, we can assume that $ n < \abs{\mu} + \mu_1$.

Notice that the condition $ n < \abs{\mu} + \mu_1$ implies that $\mu$ is either in a trivial $\stackrel{n}{\sim}$-class, or a non-minimal element in a non-trivial $\stackrel{n}{\sim}$-class.

Now, consider the category $\underline{Rep}^{ab}(S_{t=n})$; this is a rigid symmetric monoidal abelian category. From Proposition \ref{prop:obj_ab_env}, we know that $X_{\mu}$ is a projective object in $\underline{Rep}^{ab}(S_{t=n})$, while $X_{\lam}$ is a simple, but not projective object. 

Thus $X_{\mu} \otimes X_{\tau}$ is again a projective object. Decomposing it as a sum of indecomposable objects in $\underline{Rep}(S_{t=n})$, we see that all the summands must be projective objects in $\underline{Rep}^{ab}(S_{t=n})$, and so $$[X_{\mu} \otimes X_{\tau}: X_{\lambda = \lambda^{(0)}}]_{t=n} = 0$$

By Lemma \ref{lem:lifting_multip}, we have:
$$ 0 = [X_{\mu} \otimes X_{\T}: X_{\lam}]_{t=n} = \sum_{i \geq 0} (-1)^i \dim_{Frak(\bC[[T]])} \Hom_{\underline{Rep}(S_T)}(lift_{t=n}(X_{\mu}) \otimes lift_{t=n}(X_{\tau}), X_{\lambda^{(i)}}) $$

\begin{itemize}
 \item Assume $X_{\mu}$ lies in a trivial $\stackrel{n}{\sim}$-class. Then
$lift_{t=n}(X_{\mu}) = X_{\mu}$, and so 
\begin{equation}\label{eq:sum_mult_zero}
 \sum_{i \geq 0} (-1)^i \dim_{Frak(\bC[[T]])} \Hom_{\underline{Rep}(S_T)}(X_{\mu} \otimes lift_{t=n}(X_{\tau}), X_{\lambda^{(i)}}) = 0
\end{equation}

 First, assume $lift_{t=n}(X_{\T}) = X_{\T}$. In this case, Equation \eqref{eq:sum_mult_zero} becomes
 $$\sum_{i \geq 0} (-1)^i \dim_{Frak(\bC[[T]])} \Hom_{\underline{Rep}(S_T)}(X_{\mu} \otimes X_{\tau}, X_{\lambda^{(i)}}) = \sum_{i \geq 0} (-1)^{i} \bar{g}^{\lambda^{(i)}}_{\mu, \tau} = 0$$ and we are done.
 
It remains to check the case when $lift_{t=n}(X_{\T}) \neq X_{\T}$. In this case, $\T$ lies in a non-trivial $\stackrel{n}{\sim}$-class. Denote this class by $\{ \tau^{(i)} \}_{i \geq 0}$. 
 We will now prove that $$\sum_{i \geq 0} (-1)^{i} \bar{g}^{\lambda^{(i)}}_{\mu, \tau^{(j)}} = 0$$ for any $j \geq 0$ by induction on $j$.
 $$ $$
 Base: The case when $j=0$ has already been proved.
 
 Step:
 Applying Equation \eqref{eq:sum_mult_zero} to all $\tau^{(j)}$, $j \geq 1$, we get:
  \begin{align*}
   &0 = \sum_{i \geq 0} (-1)^i \dim_{Frak(\bC[[T]])} \Hom_{\underline{Rep}(S_T)}(X_{\mu} \otimes X_{\tau^{(j)}}, X_{\lambda^{(i)}}) + \\
   &+ \sum_{i \geq 0} (-1)^i \dim_{Frak(\bC[[T]])} \Hom_{\underline{Rep}(S_T)}(X_{\mu} \otimes X_{\tau^{(j)}}, X_{\lambda^{(i)}}) = \sum_{i \geq 0} (-1)^{i} \bar{g}^{\lambda^{(i)}}_{\mu, \tau^{(j)}} + \sum_{i \geq 0} (-1)^{i} \bar{g}^{\lambda^{(i)}}_{\mu, \tau^{(j-1)}}
  \end{align*}
  So assuming $$\sum_{i \geq 0} (-1)^{i} \bar{g}^{\lambda^{(i)}}_{\mu, \tau^{(j-1)}} = 0$$ we get: $$\sum_{i \geq 0} (-1)^{i} \bar{g}^{\lambda^{(i)}}_{\mu, \tau^{(j)}} = 0$$ and we are done.
  \newline
  
  \item Assume both $X_{\mu}$ and $X_{\T}$ lie in non-trivial $\stackrel{n}{\sim}$-classes, denoted by $\{\mu^{(i)} \}_{i \geq 0}$ and by $\{\T^{(i)} \}_{i \geq 0}$ respectively. Let $k, l$ be such that $\mu = \mu^{(k)}$, $\T = \T^{(l)}$. 
  
  Then 
$$ 0 = [X_{\mu} \otimes X_{\T}: X_{\lam}]_{t=n} = \sum_{i \geq 0} (-1)^i \dim_{Frak(\bC[[T]])} \Hom_{\underline{Rep}(S_T)}((X_{\mu^{(k)}} \oplus X_{\mu^{(k-1)}} ) \otimes (X_{\tau^{(l)}} \oplus X_{\tau^{(l-1)}}), X_{\lambda^{(i)}}) $$ and so 
$$\sum_{i \geq 0} (-1)^{i} \bar{g}^{\lambda^{(i)}}_{\mu^{(k)}, \tau^{(l)}} + \sum_{i \geq 0} (-1)^{i} \bar{g}^{\lambda^{(i)}}_{\mu^{(k)}, \tau^{(l-1)}} +\sum_{i \geq 0} (-1)^{i} \bar{g}^{\lambda^{(i)}}_{\mu^{(k-1)}, \tau^{(l)}} +\sum_{i \geq 0} (-1)^{i} \bar{g}^{\lambda^{(i)}}_{\mu^{(k-1)}, \tau^{(l-1)}} = 0$$

By induction on $k +l$, we can now prove that
 $$ \sum_{i \geq 0} (-1)^{i} \bar{g}^{\lambda^{(i)}}_{\mu, \tau} = (-1)^{k+l} g^{\widetilde{\lambda}(n)}_{\widetilde{\mu^{(0)}}(n), \widetilde{\tau^{(0)}}(n)}$$
 (the base case $k+l = 0$ was proved in Proposition \ref{prop:Kron_vs_reduced_Kron}).
\end{itemize}

%
\end{proof}

\end{document}